\documentclass[12pt]{amsart}

\usepackage{amsmath,amsthm,amsfonts,latexsym,amssymb,amscd,enumerate,times}
\usepackage{fullpage}
\usepackage[all]{xy}

\newtheorem{theorem}{Theorem}[section]
\newtheorem{thm}[theorem]{Theorem}
\newtheorem{prop}[theorem]{Proposition}
\newtheorem{lem}[theorem]{Lemma}
\newtheorem{cor}[theorem]{Corollary}

\theoremstyle{remark}

\theoremstyle{definition}
\newtheorem{defn}[theorem]{Definition}

\theoremstyle{definition}

\theoremstyle{remark}

\newcommand{\FillRad}{\operatorname{FillRad}}

\newcommand{\dil}{\operatorname{dil}}
\newcommand{\diam}{\operatorname{diam}}
\newcommand{\Vol}{\operatorname{Vol}}

\DeclareMathOperator{\Nil}{{\rm Nil}^3}

\DeclareMathOperator{\Hig}{Hig} 

\newcommand{\N}{\mathbb{ N}}

\newcommand{\Z}{\mathbb{ Z}}

\newcommand{\Q}{\mathbb{ Q}}
\newcommand{\R}{\mathbb{ R}}

\newcommand{\rk}{\operatorname{rk}}

\newcommand{\id}{\operatorname{id}}

{\catcode`@=11\global\let\c@equation=\c@theorem}

\date{\today}

\begin{document}

\title{Large and small group homology}

\author{Michael Brunnbauer}
\address{Mathematisches Institut, Ludwig-Maximilians-Universit\"at M\"unchen, Theresienstr.\ 39, 
D-80333 M\"unchen, Germany}
\email{michael.brunnbauer@mathematik.uni-muenchen.de}

\author{Bernhard Hanke}
\address{Zentrum Mathematik, Technische Universit\"at M\"unchen, Boltzmannstr. 3, D-85748 Garching, Germany}
\email{hanke@ma.tum.de}

\begin{abstract}
For several instances of metric largeness like enlargeability 
or having hyperspherical universal covers, we 
construct non-large vector subspaces in the rational 
homology of finitely generated groups. The functorial properties of 
this construction imply that the corresponding largeness 
properties of closed manifolds depend only on 
the image of their fundamental classes under the 
classifying map. 

This is applied to construct examples of essential manifolds 
whose universal covers are not hyperspherical, thus answering 
a question of Gromov (1986), and, more generally,  essential manifolds which are not enlargeable. 

\end{abstract}

\maketitle


\section{Introduction}

Gromov and others \cite{Gromov(1986), GL(1980), GL(1983)} introduced notions of 
largeness for Riemannian manifolds. These include enlargeability and having 
hypereuclidean or hyperspherical universal covers, and universal covers with 
infinite filling radii. 
While precise definitions are given later in this paper, we point out 
that in spite of their reference to Riemannian metrics these 
properties are independent of the chosen metrics in the case of 
closed manifolds.   

In this paper we will elaborate on the topological-homological nature of several largeness 
properties of closed manifolds and more generally of homology classes of simplicial complexes with 
finitely generated fundamental groups.

To explain some of our results in greater detail, let us here recall the definition of enlargeability
(cf. \cite{GL(1983)}).  In this paper all manifolds are assumed to be smooth and connected unless otherwise stated.  

\begin{defn} Let $M$ be a closed 
orientable manifold of dimension $n$, and let $g$ be a Riemannian metric on $M$. Then $(M,g)$ is called 
\emph{enlargeable}, 
if for every $\varepsilon>0$ there is a Riemannian cover $\overline{M}_\varepsilon \to M$ and an 
$\varepsilon$-contracting (i.\,e.\ $\varepsilon$-Lipschitz) map $\overline{M}_\varepsilon \to S^n$ 
to the $n$-dimensional unit sphere which is constant outside a compact set and of nonzero degree. 
\end{defn} 

On closed manifolds all Riemannian metrics are in bi-Lipschitz correspondence and hence the 
described property does not depend on the particular choice of $g$. However, it is important
that $g$ remains fixed for the different choices of $\varepsilon$. Examples of enlargeable manifolds 
include tori and more generally manifolds admitting Riemannian metrics of nonpositive 
sectional curvature.

\begin{defn}[\cite{Gromov(1982)}]
Let $M$ be a closed oriented manifold with fundamental class $[M]$, and let $\Phi:M\to B\pi_1(M)$ be the classifying map
of the universal cover. We call  $M$ \emph{essential} if $\Phi_*[M]\neq 0\in H_*(B\pi_1(M);\mathbb{Q})$.
\end{defn} 

In \cite{HS(2006)} it was shown by index theoretic methods that 
enlargeable manifolds are essential, if the cover $\overline{M}_{\varepsilon} \to M$ can always be 
assumed to be finite. 
Relying on ideas from coarse geometry, \cite[Corollary 1.3]{HKRS(2007)} states
essentialness for manifolds with  
hyperspherical universal covers (these are manifolds where $\overline{M}_\varepsilon$ 
may always be chosen as the universal cover).  Elaborating on the methods from \cite{HS(2006)}, 
essentialness of all enlargeable manifolds was proven in \cite{HS(2007)}.

Extending the results of \cite{HKRS(2007), HS(2006), HS(2007)} we will show that manifolds which 
are either enlargeable or have universal covers which are (coarsely) hyperspherical, (coarsely) hypereuclidean or 
macroscopically large (these notions will be defined in Section \ref{largeriem}) are essential. 
Our approach is independent of index theory or coarse geometry (unless
the largeness condition under consideration refers to coarse notions). 

We will moreover prove that each of these largeness properties
depends only on the image $\Phi_*[M] \in H_*(B \pi_1(M) ; \mathbb{Q})$ 
of the fundamental class under the classifying map. This property may be  
called \emph{homological invariance of largeness}. 

Both conclusions are implied by the following theorem which will be proved in Section \ref{essential_and_homol_inv}
of our paper. 

\begin{thm}\label{intro_p4_subsp_thm}
Let $\Gamma$ be a finitely generated discrete group and let $n$ be a natural number. 
Let $P$ denote one of the properties of being enlargeable or having a universal cover which 
is (coarsely) hyperspherical, (coarsely) hypereuclidean or macroscopically large, there is a rational vector subspace 
\[
   H_n^{\rm sm(P)}(B\Gamma; \mathbb{Q}) \subset H_n(B\Gamma;\mathbb{Q})
\]
of ``non-large'' (i.\,e.\ small) classes with respect to $P$.  In the case when $\Gamma$ is finitely presentable, the following holds: 
If $M$ is a closed oriented $n$-dimensional manifold with fundamental 
group $\Gamma$ and classifying map $\Phi:M\to B \Gamma$, then 
$\Phi_* [M] \in H_n^{\rm sm(P)}(B\Gamma;\Q)$ if and only if $M$ is not large in the respective sense.
\end{thm}

In Definition \ref{enlhom}  we will introduce largeness properties in general 
for homology classes of connected simplicial complexes with finitely generated fundamental groups
by adapting the classical definitions accordingly. The hard part of 
this approach lies in the verification that the required properties depend on the 
given  homology class only. In this respect  homological 
invariance of largeness is built into the definition right away. Once
this has been achieved it will be easy to show that the classes which are ``small'', i.\,e.\ not large in the respective sense, 
enjoy the nice algebraic property of forming a vector subspace, see Theorem \ref{subsp_thm}.   

If the simplicial complex in question is the classifying space of a finitely generated group, this 
approach emphasizes our point of view that the largeness properties in Theorem \ref{intro_p4_subsp_thm}
should be regarded as metric properties of finitely generated groups (respectivelly their rational homology) 
much like the quasi-isometry type of the word metric itself. 

Homological invariance of largeness for the classical 
case of closed manifolds is a simple consequence of the functorial properties 
for large homology classes proven in Proposition \ref{funktoreins}. Together with 
the fact that the non-large classes form a vector subspace (and hence contain the class $0$ in each degree), 
this shows  indeed that enlargeable manifolds and more generally manifolds with 
(coarsely) hyperspherical, (coarsely) hypereuclidean or macroscopically large universal 
covers are essential. 

In Section \ref{bspiele} we will illustrate by examples that the subspaces $H_*^{\rm sm(P)}$ 
are in general different from zero and may even 
depend on the specific largeness property $P$. In particular we will prove the following 
consequence of Theorem \ref{intro_p4_subsp_thm}:

\begin{thm} \label{example1} For all $n \geq 4$ there are enlargeable (hence essential) manifolds of dimension $n$ 
whose universal covers are neither (coarsely) hyperspherical, (coarsely) hypereuclidean nor macroscopically large. 
\end{thm}

Gromov asked in \cite[page 113]{Gromov(1986)}  whether universal covers of essential 
manifolds were always hyperspherical. Theorem \ref{example1} gives a negative answer. It also 
shows the remarkeable fact that enlargeable manifolds do not always have hyperspherical 
universal covers. This provides a late justification for the organisation of an 
argument in  \cite{HKRS(2007)}: In that paper the proof that enlargeable manifolds are Baum-Connes essential 
was much easier for manifolds with hyperspherical universal covers (these 
were called {\em universally enlargeable} in \cite{HKRS(2007)}) than for 
general enlargeable manifolds. Now we see that the general case cannot be 
reduced to the case of manifolds with hyperspherical universal covers.    

By a refinement of our methods we also get the following result, showing that even the 
more flexible property of enlargeability is not always implied by essentialness: 

\begin{thm} \label{example2} For all $n \geq 4$, there are $n$-dimensional closed manifolds which are essential, but not 
enlargeable. 
\end{thm} 

Because arbitrary covers of these manifolds need to be controlled, the proof of this  
result is technically much harder than the one of Theorem \ref{example1}. Our 
argument makes essential use of the Higman $4$-group \cite{Hig}. 
 
These examples are interesting because enlargeability is the 
most flexible of the largeness properties considered in this paper 
in terms of which coverings of the given manifold can be used. 
In this respect they point to a principal limitation of the use  of 
largeness of closed manifolds for proving the strong Novikov conjecture in full generality,
cf. \cite{HKRS(2007), HS(2006)}. 

We remark however that none of our 
examples appearing in Theorem \ref{example1} or \ref{example2} is aspherical. 
In the case of aspherical manifolds there is apparently room 
to use metric largeness properties in order to prove general theorems on the non-existence 
of positive scalar curvature metrics and related properties, notably if the 
fundamental group has finite asymptotic dimension \cite{Dranish1, Dranish2}. 

The characterization of metric largeness by certain subspaces 
of group homology has been remarked before in the context of vanishing simplicial 
volume \cite[Section 3.1]{Gromov(1982)} and -- in some more restricted setting -- in the context of 
positive scalar curvature metrics on high-dimensional manifolds with 
non-spin universal covers \cite{RosenbergStolz}.  Recently \cite{Br} the first author 
of the present article showed that the systolic constant, the minimal volume entropy, and the 
spherical volume of closed manifolds only depend on the image of their fundamental 
classes in the integral homology of their fundamental groups under the classifying maps 
of their universal covers.  

In the definition of enlargeability  the maps $\overline{M}_\varepsilon\to S^n$ are assumed to contract distances. 
If they are only required to contract volumes of $k$-dimensional submanifolds, 
then $M$ is called \emph{$k$-enlargeable}. In the case $k=2$, this property is also called 
\emph{area-enlargeability} \cite{GL(1983)}.

Relying on index theory it was shown in \cite{HS(2007)} that area-enlargeable manifolds are  
essential. In Section \ref{higher_enl} of the paper at hand we will prove the following 
more general statement by methods similar to those employed in Section \ref{essential_and_homol_inv}.

\begin{thm}\label{intro_p4_higher_enl_impl_ess}
Let $M$ be a closed oriented $n$-dimensional manifold. If $M$ is $k$-enlargeable and satisfies
\[ 
    \pi_i(M)=0 \;\;\;\mbox{for}\;\;\; 2\leq i\leq k-1, 
\]
then $M$ is  essential. In particular, area-enlargeable manifolds are  essential.
\end{thm}

For $k\leq 2$ the condition on the homotopy groups is to be understood as empty. Note that for $k>2$ the condition 
is in fact necessary: 
Let $M$ be an enlargeable manifold. Then the product $M\times S^2$ is $3$-enlargeable, but the classifying map 
$M\times S^2 \to B\pi_1(M)$ 
sends the fundamental class to zero, i.\,e.\ $M\times S^2$ is not  essential. This is in accordance
with Theorem \ref{intro_p4_higher_enl_impl_ess} as $\pi_2(M \times S^2) \neq 0$. 

For $k\geq n+1$ the $k$-dimensional volume of any subset of $S^n$ is zero, of course. Therefore, in this case 
the assumption of $k$-enlargeability in Theorem \ref{intro_p4_higher_enl_impl_ess} can be dropped and 
the remaining nontrivial requirement is $\pi_i(M)=0$ for $2\leq i \leq k-1$. The inequality
$k \geq n+1$ and the Hurewicz theorem then imply
that all homotopy groups of the universal cover of $M$ vanish.  
In other words, Theorem \ref{intro_p4_higher_enl_impl_ess} includes the well known statement that aspherical 
manifolds are essential.

Hence the conditions in Theorem \ref{intro_p4_higher_enl_impl_ess} interpolate between two extreme cases: 
enlargeable and area-enlargeable manifolds on the one side and aspherical ones on the other side. 
It was shown in  \cite{GL(1983)} that area-enlargeable spin manifolds do not carry metrics of positive scalar curvature 
and it has been conjectured that the conclusion is valid for all 
aspherical manifolds. In this respect it seems reasonable to  conjecture that the conditions in Theorem
\ref{intro_p4_higher_enl_impl_ess} are also obstructions to the existence of 
positive scalar curvature metrics. In fact the strong Novikov conjecture implies 
that essential spin manifolds do not admit positive scalar curvature 
metrics, cf. \cite{Rosenberg(1983)}.

Concerning homological invariance, $k$-enlargeability for $k \geq 2$ seems 
to bahave less favourably than the other largeness properties considered in our work. However, 
we will not pursue this question further in this paper. Theorem \ref{intro_p4_higher_enl_impl_ess} is related in 
spirit to Theorem 2.5 in \cite{BrKo} which deals with functorial properties of hyperbolic cohomology classes. 

This article is based in part on a chapter of the first author's thesis \cite{Br2}. In particular, homological 
invariance of largeness properties (which is part of Theorem \ref{intro_p4_subsp_thm} of the present article) and 
 Theorem \ref{intro_p4_higher_enl_impl_ess} were first proved there. The most important novelties 
 of the present paper are a more systematic treatment of large homology classes  
 in Section \ref{essential_and_homol_inv} and - based on that - 
 the construction of  interesting examples of enlargable manifolds without hyperspherical 
 universal covers (Theorem \ref{example1}) and of essential manifolds that are not enlargeable (Theorem \ref{example2}). 

{\em Acknowledgements:} 
The first author would like to thank his thesis advisor D.\,Kotschick for continuous support and 
encouragement. Both authors gratefully acknowledge financial support from the 
\emph{Deutsche Forschungsgemeinschaft}, useful comments by D.\,Kotschick concerning 
a preliminary version of this paper and numerous helpful remarks by the referee.

\section{Large Riemannian manifolds} \label{largeriem} 

In this section we will recall classical notions of metric largeness for Riemannian manifolds, most 
of which were first formulated by Gromov, see for example \cite{Gromov(1986), Gromov(1996a), GL(1983)},  and also 
\cite{Cai(1994), Guth(2006p2)}. They include 
the properties of being
\begin{itemize}
   \item $k$-hypereuclidean and $k$-hyperspherical, see Definition \ref{erste}, 
   \item $k$-enlargeable, see Definition \ref{defnenl},
   \item having infinite filling radius, see Definition \ref{inffill} and 
   \item being macroscopically large, see Definition \ref{macrlarge}. 
\end{itemize}
In Proposition \ref{charballoon} we will characterise hypersphericy in terms of  the existence 
of a Lipschitz map of nonzero degree to the {\em balloon space}, which 
was introduced in \cite{HKRS(2007)}. Upon passing from Lipschitz to large scale Lipschitz maps this allows us to define 
coarsely hyperspherical manifolds, a notion similar to coarsely hypereuclidean manifolds, which by 
definition admit coarse maps to Euclidean space of nonzero degree. 
Furthermore we will discuss several implications between these largeness properties. 
In particular, in Proposition \ref{fill=macro} we will show that the classes of macroscopically large manifolds and 
of manifolds with infinite filling radii coincide. This is remarkable because it relates a coarse (i.\,e.\ quasi-isometric)  
property to a bi-Lipschitz one.

Let $f:(M,g_M)\to (N,g_N)$ be a smooth map between (not necessarily compact) Riemannian manifolds, 
and let $k$ be a positive integer.

\begin{defn}\label{dilation}
The \emph{$k$-dilation} of $f$ is defined as
\[ 
         \dil_k(f) := \sup_{p\in M} \| \Lambda^k Df_p \| \in \R \cup \{\infty\} 
 \]
the supremum of the norms of the $k$-fold exterior product of the differential $Df$. 
\end{defn}

Said differently, the $k$-dilation is the smallest number $\varepsilon$ such that for any $k$-dimensional 
submanifold $A \subset M$ the $k$-dimensional volume $\Vol_k(f(A))$ of the image $f(A)\subset N$ is bounded by 
$\varepsilon\cdot\Vol_k(A)$. The $1$-dilation is the smallest Lipschitz constant for $f$.

Let $p\in M$ be a point, and let $n$ be the dimension of $M$. Denote by $\lambda_1\geq\ldots\geq \lambda_n >  0$ the 
eigenvalues of the Gram matrix of the pullback $(Df_p)^*(g_N)_{f(p)}$ with respect to $(g_M)_p$. Then $\|\Lambda^k Df_p\| = 
\lambda_1\cdot\ldots\cdot\lambda_k$. Therefore, the inequality 
\[
  \dil_l(f)^{1/ l} \leq \dil_k(f)^{1/k} 
\]
holds for all $l\geq k$.

Let $(V,g)$ be a complete orientable Riemannian manifold of dimension $n$. A choice of orientation for 
$V$ defines a fundamental class $[V]\in H_n^{\rm lf}(V;\mathbb{Z})$ in locally finite homology. In  this context the 
 mapping degree
is well-defined for proper maps to oriented manifolds and for maps to closed oriented manifolds $Z$ that 
are {\em almost proper}, i.e. constant outside a compact set. This can be made rigorous by 
adding an infinite whisker to $Z$, extending the given almost proper map to a proper map with 
target $Z \cup {\rm whisker}$, and observing that $H^{\rm lf}_*(Z \cup {\rm whisker}) = \widetilde H_*(Z)$. 

In the following, we equip Euclidean spaces and unit spheres with their standard metrics.  

\begin{defn} \label{erste}
We call $(V,g)$ \emph{$k$-hypereuclidean} if there is a proper map
\[ 
   f:V \to\mathbb{R}^n
\]
of nonzero degree and of finite $k$-dilation. 
It is called \emph{$k$-hyperspherical} if for every $\varepsilon>0$ there is an almost proper map
\[ 
     f_\varepsilon : V \to S^n
\]
of nonzero degree such that $\dil_k(f_\varepsilon)\leq\varepsilon$. 
For $k=1$ we will omit the number, and for $k=2$ we will speak of \emph{area-hypereuclidean} and 
\emph{area-hyperspherical} manifolds.
\end{defn}

By the above inequality, every $k$-hypereuclidean or $k$-hyperspherical manifold is also 
$l$-hyper\-euclidean respectively $l$-hyperspherical for any $l \geq k$. Since $\mathbb{R}^n$ 
is obviously hyperspherical, any $k$-hypereuclidean manifold is also $k$-hyper\-spherical. Note also that
both notions depend only on the bi-Lipschitz type of the metric $g$.

Closely related is the notion of enlargeability. It was introduced by Gromov and 
Lawson in \cite{GL(1980)} and in the following more general form in \cite{GL(1983)}.

\begin{defn} \label{defnenl}
An orientable $n$-dimensional manifold $V$ is called \emph{$k$-enlargeable} if for every complete 
Riemannian metric $g$ on  $V$ and every $\varepsilon>0$ there is a Riemannian cover $\overline{V}_\varepsilon$ 
of $V$ and an almost proper map 
\[ 
   f_\varepsilon: \overline V_\varepsilon \to S^n
\]
of nonzero degree such that $\dil_k(f_\varepsilon)\leq\varepsilon$. 
As before, we will omit the number $k$ in the case $k=1$ and speak of \emph{area-enlargeable} manifolds in the case $k=2$.
\end{defn}

If $V$ is closed, then all Riemannian metrics on $V$ are bi-Lipschitz to each other and 
it is enough that $V$ satisfies the above 
condition with respect to one Riemannian metric. 

The significance of the notion of enlargeability is demonstrated by the following theorem, see 
\cite[Theorem 6.1]{GL(1983)}. 

\begin{thm}
If $V$ is area-enlargeable and the covers $\overline{V}_\varepsilon$ in Definition \ref{defnenl} 
may be chosen spin, then $V$ does not carry a complete Riemannian metric of uniformly positive scalar curvature.
\end{thm}

Next, we will revisit the notion of filling radius. Recall that every Riemannian metric $g$ on $V$ induces a 
path metric $d_g$ on $V$. Denote by $L^\infty(V)$ the vector space of all functions $V \to \R$ with the uniform `norm' 
$\|-\|_\infty$. This is not a norm proper since it may take infinite values. Therefore the induced `metric' 
is not an actual metric. Nevertheless, the \emph{Kuratowski embedding}
\begin{align*} 
\iota_g :(V,d_g) &\hookrightarrow L^\infty(V), \\
v &\mapsto d_g(v,- )
\end{align*}
is an isometric embedding by the triangle inequality. 

One could replace $L^\infty(V)$ by its affine subspace $L^\infty(V)_b$ that is parallel to the Banach space of all
bounded functions on $V$ and contains the distance function $d_g(v,-)$ for some $v\in V$. Then the image of 
the Kuratowski embedding is contained in $L^\infty(V)_b$, and the `norm' $\|-\|_\infty$ induces an actual 
metric on $L^\infty(V)_b$. Since all points of $L^\infty(V)$ outside of this affine subspace are already 
infinitely far away from it, 
this would not change the following definition.

\begin{defn} \label{inffill}
The \emph{filling radius} of $(V,g)$ is defined as
\[ 
   \FillRad(V,g) := \inf\{ r>0 ~|~ \iota_{g*}[V] = 0 \in H_n^{\rm lf} (U_r(\iota_g V);\mathbb{Q})\} 
\]
where $U_r(\iota_g V)\subset L^\infty(V)$ denotes the open $r$-neighborhood of the image $\iota_g V \subset L^{\infty}(V)$. 
If the set on the right hand side is empty, we say that $(V,g)$ has {\em infinite filling radius}.  
\end{defn}

Note that for closed manifolds $L^\infty(V)_b$ is the vector space of all bounded functions on $V$ and 
the above definition of the filling radius coincides with the usual definition (see \cite{Gromov(1983)}, 
Section 1). For noncompact manifolds the filling radius need not be finite. For instance
the filling radius of the Euclidean space is infinite.

It follows from the definition that the property of having infinite filling radius depends only 
on the bi-Lipschitz type of the metric $g$.

We recall the following implication that 
was shown in \cite{Gromov(1986)} (see also \cite{Cai(1994)}).

\begin{prop} \label{hypspher-fill}
If $(V,g)$ is hyperspherical, then its filling radius is infinite.
\end{prop}

We have seen that hypereuclidean manifolds are hyperspherical and that
hyperspherical manifolds have infinite filling radius. It is not known whether these 
implications are equivalences or not.

In \cite{GY(2000)}  Gong and Yu used coarse algebraic topology to define another notion of largeness, 
which is thought to be closely related to Gromov's definitions. In fact we will show that it is equivalent 
to the property of having infinite filling radius.

First we will recall some basics on coarse geometry. For more details we refer to
 Roe's book \cite{Roe(2003)}, in  particular to Chapter 5 on coarse algebraic topology. 

Let $X$ be a metric space. 
A cover $\mathcal{U}$ of $X$ is called \emph{uniform} if the diameters 
of its members are uniformly bounded and if every bounded set in $X$ meets only finitely many members of 
$\mathcal{U}$. A familiy $(\mathcal{U})_{i\in I}$ of uniform covers is called \emph{anti-\v{C}ech system} 
if for every $r>0$ there exists a cover $\mathcal{U}_i$ with Lebesgue number at least $r$.

The \emph{nerve} of a cover $\mathcal{U}$ will be denoted by $|\mathcal{U}|$. It is the simplicial complex whose simplices 
are finite subsets of $\mathcal{U}$ with nonempty intersection in $X$. In particular the set of vertices is equal to 
$\mathcal{U}$.
The nerve of a uniform cover is locally finite. 

If $\mathcal{U}$ and $\mathcal{V}$ are two uniform covers such that the Lebesgue number of $\mathcal{V}$ is bigger 
than the uniform bound on the diameters of the sets of $\mathcal{U}$, we write $\mathcal{U} \leq \mathcal{V}$. In this 
way the set of uniform covers of $X$ becomes directed. If $\mathcal{U} \leq \mathcal{V}$, then there is a proper simplicial 
map $|\mathcal{U}|\to |\mathcal{V}|$ mapping each vertex $U\in\mathcal{U}$ to some vertex $V\in\mathcal{V}$ 
that contains $U$. The proper homotopy class of this map 
is uniquely determined.

If $X$ is proper (i.e. bounded subsets are precompact) then 
anti-\v{C}ech systems alway exist. Given an anti-\v{C}ech system 
one defines the \emph{coarse homology} of $X$ as
\[ 
H\!X_k (X;\mathbb{Q}) := \varinjlim H_k^{\rm lf}(|\mathcal{U}_i|;\mathbb{Q}). 
\]
This is independent of the choice of the anti-\v{C}ech system.

For proper $X$ and for any uniform cover $\mathcal{U}$ 
there is a proper map $X\to |\mathcal{U}|$ that sends each point $x\in X$ to a point in the simplex spanned 
by those $U\in\mathcal{U}$ that contain $x$. Moreover, the proper homotopy class of such a map is 
uniquely determined. Therefore, one gets an induced homomorphism
\[ 
    c: H_k^{\rm lf}(X;\mathbb{Q}) \to H\!X_k (X;\mathbb{Q}), 
\]
which will be called the \emph{character homomorphism} of $X$.

\begin{defn}[\cite{GY(2000)}] \label{macrlarge}
A complete oriented $n$-dimensional Riemannian manifold $V$ is called \emph{macroscopically large}, if
\[ 
  c[V] \neq 0 \in H\!X_n(V;\mathbb{Q}). 
\]
\end{defn}

Note that this property depends only on the quasi-isometry class of the metric. We will show that 
macroscopic largeness is equivalent to having infinite filling radius. This proves
that the property of having infinite filling radius depends only on the quasi-isometry class of the Riemannian metric. Note that the quasi-isometry class strictly includes the bi-Lipschitz class. It is not known whether hypereuclideaness or hypersphericity are 
also invariant under quasi-isometries.

\begin{prop}\label{fill=macro}
Let $V$ be a complete orientable Riemannian manifold. Then $V$ is macroscopically 
large if and only if its filling radius is infinite.
\end{prop}

For the proof we will need the notion of a coarse map. The general definition is a bit involved. 
But recall from \cite[Section 1.3]{Roe(2003)} that a map $f:X\to Y$ from a path metric space to a 
metric space is \emph{coarse} if and only if it is large scale Lipschitz and (metrically)  proper.

\begin{proof}
Identify $V$ with its image under the Kuratowski embedding, and let $n$ be the dimension of $V$.

First assume that $\FillRad(V,g)<r$ for some finite $r$. Then there is a locally finite complex $X\subset U_r(V)$
containing $V$ such that $[V]=0\in H_n^{\rm lf}(X;\mathbb{Q})$. Moreover, the inclusion $V\hookrightarrow X$ is a 
coarse equivalence since the coarse map that assigns to a point $x\in X$ a point $v\in V$ with 
$d(x,v)\leq r$ is an inverse. The commutative diagram
\[
     \xymatrix{
H_n^{\rm lf}(V;\mathbb{Q}) \ar[r] \ar[d]_-c & H_n^{\rm lf}(X;\mathbb{Q}) \ar[d]^-c \\
H\!X_n(V;\mathbb{Q}) \ar[r]^-\cong & H\!X_n(X;\mathbb{Q})
}
\]
shows that $c[V]=0\in H\!X_n(V;\mathbb{Q})$, i.\,e.\ $(V,g)$ is not macroscopically large. 
(Note that $U_r(V)$ is also coarsely equivalent to $V$ but it is not proper. Therefore, it is not clear whether 
it admits a character homomorphism.)

To prove the converse implication, assume that $(V,g)$ is not macroscopically large. By the definition of 
the direct limit there is a uniform cover $\mathcal{U}$ of $V$ such that 
$\phi_*[V]=0\in H_n^{\rm lf}(|\mathcal{U}|;\mathbb{Q})$ where $\phi:V\to|\mathcal{U}|$ is a proper map that 
sends each point $v\in V$ to a point in the simplex spanned by those $U\in\mathcal{U}$ that contain $v$. 
Let $r>0$ be an upper bound on the diameters of the sets of $\mathcal{U}$. 

Define a map $\psi: |\mathcal{U}|\to L^\infty(V)$ by sending each vertex $U\in\mathcal{U}$ to 
some point $\psi(U)\in U\subset V$ and by extending this linearly over each simplex of the nerve. 

Let $p$ be a point in $|\mathcal{U}|$. It may be written as $p=\sum  \lambda_i U_i$ with $\sum \lambda_i = 1$, 
$\lambda_i>0$, and $U_i\in\mathcal{U}$ such that $\bigcap U_i\neq\emptyset$. Then $\psi(p)=\sum \lambda_i\psi(U_i)$ and
\begin{align*}
 d(\psi(p), \psi(U_1)) &= \left\|\sum \lambda_i \psi(U_i)-\psi(U_1) \right\|_\infty \\
&\leq \sum \lambda_i \|\psi(U_i)-\psi(U_1)\|_\infty \\
&\leq 2r
\end{align*}
since $U_i\cap U_1\neq\emptyset$ for all $i$. 
This shows that the image of $\psi$ lies in the $2r$-neighborhood of $V$ in $L^\infty(V)$. Hence
\[ (\psi\circ\phi)_*[V] = 0 \in H^{\rm lf}_n(U_{2r}(V);\mathbb{Q}). \]

Let $v\in V$ be a point. Say $v$ lies in the sets $U_1,\ldots,U_m\in\mathcal{U}$ and in no other set of $\mathcal{U}$.
 Then $\phi(v) = \sum_{i=1}^m \lambda_i U_i$ for some $\lambda_i\geq 0$ with $\sum\lambda_i=1$. Therefore
\begin{align*}
d(\psi(\phi(v)),v) &= \left\|\sum\lambda_i \psi(U_i) -v \right\|_\infty \\
&\leq \sum\lambda_i \|\psi(U_i)-v\|_\infty \\
&\leq r
\end{align*}
since $v\in U_i$ for all $i = 1, \ldots, m$. Thus the linear homotopy from the inclusion $V\hookrightarrow L^\infty(V)$ to $\psi\circ\phi$ 
is proper and lies entirely in $U_r(V)$. We conclude
\[ 
      [V] = (\psi\circ\phi)_*[V] \in H_n^{\rm lf} (U_{2r}(V);\mathbb{Q}), 
\]
and consequently $[V]=0\in H_n^{\rm lf}(U_{2r}(V);\mathbb{Q})$, hence $\FillRad(V,g)\leq 2r <\infty$.
\end{proof}

Propositions  \ref{hypspher-fill} and \ref{fill=macro} show that complete hyperspherical manifolds are macroscopically large. 
If the given hyperspherical manifold is the universal cover of a closed manifold,   this is proved directly in \cite[Proposition 3.1]{HKRS(2007)} using the \emph{balloon space} 
$B^n$. This path metric space is defined as a real half-line $[0,\infty)$ with an $n$-dimensional round sphere $S^n_i$ 
of radius $i$ attached at each positive integer $i\in[0,\infty)$.

\begin{prop}[\cite{HKRS(2007)}, Proposition 2.2]
The $n$-dimensional coarse homology of the balloon space is given by
\[
 H\!X_n(B^n;\mathbb{Q}) = \left(\prod_{i=1}^\infty \mathbb{Q}\right) / \left(\bigoplus_{i=1}^\infty \mathbb{Q}\right). 
\]
Moreover, for the  locally finite homology we have $H_n^{\rm lf}(B^n;\mathbb{Q}) = \prod_{i=1}^\infty \mathbb{Q}$, 
and the character homomorphism $c:H_n^{\rm lf}(B^n;\mathbb{Q})\to H\!X_n(B^n;\mathbb{Q})$ is the canonical projection.
\end{prop}

Using this computation we obtain the following characterization of hyperspherical manifolds.

\begin{prop} \label{charballoon}
A complete oriented Riemannian manifold $V$ of dimension $n$ is hyperspherical, if and only if 
there exists a proper Lip\-schitz map $f:V\to B^n$ such that $f_*[V]\neq 0 \in H\!X_n(B^n;\mathbb{Q})$.
\end{prop}

\begin{proof}
First, assume that $V$ is hyperspherical. We will construct a sequence of closed balls 
\[ 
      \emptyset=B_0\subset B_1\subset B_2\subset\ldots\subset V 
\]
that exhausts $V$ and a sequence of $1$-Lipschitz maps $f_i:B_i\setminus \mathring{B}_{i-1}\to S^n_i \vee[i,i+1]\subset 
B^n$ such that $f_i(\partial B_{i-1})=i$, $f_i(\partial B_i)=i+1$, and such that $f_i$ is of nonzero degree 
as a map to $S^n_i$. 

Assume that the balls and maps have been constructed up to index $i-1$. Let $S^n_R$ be the 
round sphere of radius $R$ with a large $R$ which will be specified later. Choose a $1$-Lipschitz map $f'_i:V\to S^n_R$ 
that is constant outside a compact set $K_i$ and that is of nonzero degree. 
Without loss of generality, we may assume that $B_{i-1}\subset K_i$ and that $f'_i(B_{i-1})$ 
and the point $f'_i(V\setminus K_i)$ avoid a ball of radius $\pi i$ in $S^n_R$ 
(choose for instance $R\geq 2 i+r/\pi$ where $r$ is the radius of $B_{i-1}$). 
Let $g_i:S^n_R\to S^n_i$ be a nonexpanding map that contracts everything outside this ball 
of radius $\pi i$ to the base point of $S^n_i$.

Choose a ball $B_i\subset V$ such that $K_i\subset B_i$ and such that $d(\partial B_i, K_i)\geq 1$. 
Define $f_i$ as follows:
\[ 
f_i(v) := \begin{cases}
g_i\circ f'_i(v) & \mbox{for } v\in B_i\setminus \mathring{B}_{i-1} \mbox{ and }  d(v,\partial B_i)\geq 1 \\
i+1-d(v,\partial B_i) & \mbox{for } v\in B_i, d(v,\partial B_i)\leq 1
\end{cases}
\]
Then $f_i$ has the asserted properties.

The collection of the maps $f_i$ defines a proper $1$-Lipschitz map $f:V\to B^n$ such that every entry 
of $f_*[V]\in H_n^{l\!f}(B^n;\mathbb{Q})\cong\prod_{i=1}^\infty \mathbb{Q}$ is nonzero.  
In particular $f_*[V]\neq 0 \in H\!X_n(B^n;\mathbb{Q})$ as required.

For the converse, let $f:V\to B^n$ be a proper Lipschitz map such that $f_*[V]\neq 0 \in H\!X_n(B^n;\mathbb{Q})$. 
Let $\varepsilon>0$, and choose an integer $i\geq \dil_1(f)/\varepsilon$ such that the $i$-th entry of 
$f_*[V]\in H_n^{\rm lf}(B^n;\mathbb{Q})\cong\prod_{i=1}^\infty 
\mathbb{Q}$ is not zero. This is possible since by assumption there are infinitely many nonvanishing entries.

Let $f_\varepsilon$ be the composition of $f$ with the canonical quotient map from $B^n$ to the $i$-th sphere $S^n_i$ 
and the dilation from this sphere of radius $i$ to the unit sphere. Then $f_\varepsilon$ is constant outside a compact set, 
has nonzero degree, and its dilation is given by $\dil_1(f)/i\leq\varepsilon$. This proves that $V$ is indeed hyperspherical.
\end{proof}

In \cite[Proposition 3.1]{HKRS(2007)} it is shown that hyperspherical universal covers of closed Riemannian 
manifolds admit proper Lipschitz maps to the balloon space sending the locally finite fundamental class
of the universal covers to nonzero classes. The corresponding implication in our
Proposition \ref{charballoon} is slightly more general in that it is not assumed that the metric on $V$ is invariant under a 
cocompact group action.

\begin{defn} \label{coarsestuff}
A complete oriented Riemannian manifold $V$ of dimension $n$ is called \emph{coarsely hypereuclidean}, if there is a coarse map
\[ 
f:V  \to \mathbb{R}^n
\]
such that $f_*[V]\neq 0 \in H\!X_n(\mathbb{R}^n;\mathbb{Q})\cong \mathbb{Q}$. 
It is called \emph{coarsely hyperspherical} if there is a coarse map
\[
   f: V  \to B^n 
\]
to the balloon space such that $f_*[V]\neq 0 \in H\!X_n(B^n;\mathbb{Q})$.
\end{defn}

These two notions depend only on the quasi-isometry class of the metric on $V$. 

The following diagram summarizes the known implications between some  of the largeness properties 
on a complete Riemannian manifold discussed in this  section. 
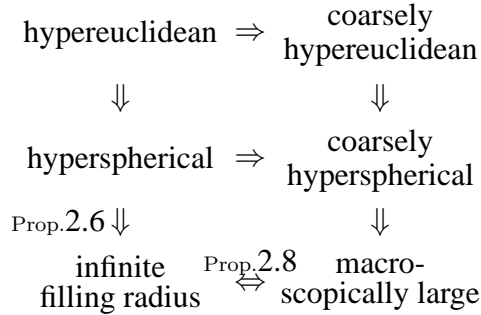
\[\xymatrix{
*\txt{hypereuclidean} \ar@{}[r]|*{\Rightarrow} \ar@{}[d]|*{\Downarrow} & *\txt{coarsely\\hypereuclidean} \ar@{}[d]|*{\Downarrow} \\
*\txt{hyperspherical} \ar@{}[r]|*{\Rightarrow} \ar@{}[d]|*{\Downarrow}_{\rm Prop. 
\ref{hypspher-fill}\;} & *\txt{coarsely\\hyperspherical} \ar@{}[d]|*{\Downarrow} \\
*\txt{infinite\\filling radius} \ar@{}[r]|*{\Leftrightarrow}^{\rm Prop. \ref{fill=macro}} & *\txt{macro-\\scopically large}
}\]

As $H_n^{\rm lf}(\R^n;\Q) \cong H\!X_n(\R^n;\Q) \cong \Q$,   hypereuclidean manifolds are coarsely 
hypereuclidean, and Proposition \ref{charballoon} implies that hyperspherical manifolds are coarsely hyperspherical. 
This explains the two upper horizontal arrows. 
Moreover, the proof of Proposition \ref{charballoon} shows the existence of a coarse map $\R^n \to B^n$ sending the coarse 
fundamental class of $\R^n$ to a non-zero class in $H\!X_n(B^n;\Q)$. This implies the upper vertical arrow on the right.
The lower right vertical implication follows by the very definition of macroscopic largeness. 
Apart from the lower horizontal arrow it is not known if any of the implications is an equivalence. 

We also remark that the properties on the left-hand side are invariants of the bi-Lipschitz class of the given metric, the ones 
on the right-hand side of its quasi-isometry class.

\section{Largeness in homology}\label{essential_and_homol_inv}

In this section we shall formulate the concept of  largeness for 
rational homology classes in simplicial complexes with finitely generated
fundamental groups. We will then prove Theorem \ref{intro_p4_subsp_thm}.
In this section, the term {\em large} is a placeholder for one of the properties of being enlargeable or having a 
universal cover which is (coarsely) hyperspherical, (coarsely) hypereuclidean or macroscopically 
large. 

The method of extending differ\-ential-geometric concepts from smooth manifolds 
to more general spaces like simplicial complexes has occured at other places in  the literature in similar contexts, 
see e.\,g.\ \cite{Br}. 

We equip the $n$-dimensional simplex $\Delta^n$ with the metric induced by the 
standard embedding into $\R^{n+1}$. Recall from \cite[Chapter 1]{Gromov(2001)} that each connected simplicial complex 
comes with a canonical path metric restricting to this standard metric 
on each simplex. Furthermore, if $p : X \to Y$ is a covering map of path connected 
topological spaces and $Y$ is equipped with a path metric, then 
there is a unique path metric on $X$ so that $p$ is a local isometry. If 
$Y$ is a simplicial complex with the canonical path metric and 
$X$ carries the induced simplicial structure, then this is the canonical path metric on $X$.

A connected subcomplex $S$ of a connected simplicial complex $X$ is called {\em $\pi_1$-surjective}, if 
the inclusion induces a surjection of fundamental groups and we say that $S$ {\em carries} a 
homology class $c \in H_*(X;\Q)$, if $c$ is in the image of the map in homology induced 
by the inclusion.  

If $p : \overline{X} \to X$ is a (not necessarily connected) cover of a simplicial complex $X$ and $c \in H_{n}(X;\Q)$ 
is a simplicial homology class, the {\em transfer} $p^!(c) \in H_n^{\rm lf}(\overline X;\Q)$ is 
represented by the formal sum of all preimages of simplices in a chain representative of $c$, with appropriate coefficients.

\begin{defn} \label{enlhom} Let $X$ be a connected simplicial complex with finitely generated fundamental group and 
let $c \in H_n(X;\Q)$ be a (simplicial) homology class. Choose a finite connected $\pi_1$-surjective subcomplex 
$S \subset X$ carrying $c$.  (This subcomplex exists, because $\pi_1(X)$ is finitely generated.) 

The class $c \in H_n(X;\Q)$ is called {\em enlargeable}, if the following holds: 
Let $\varepsilon>0$. Then there is a connected cover $p : \overline X \to X$ 
and an almost proper $\varepsilon$-contracting map 
$f_\varepsilon: \overline S  \to S^n$ which sends the class $p^!(c)$ to a nonzero class in $\widetilde H_n(S^n;\Q)$.
Here $\overline S := p^{-1}(S)$ (which is connected as $S$ is $\pi_1$-surjective) 
is equipped with the canonical path metric.    

The class $c$ is called \emph{(coarsely) hypereuclidean}, \emph{(coarsely) hyperspherical}, 
respectively \emph{macroscopically large} if the complex $\widetilde S = p^{-1}(S)$ 
associated to the universal cover $p : \widetilde X \to X$ together with the transfer class $p^!(c)$ 
enjoys the according property. 
\end{defn} 

It is important to work with $\pi_1$-surjective subcomplexes $S$, because then the covers $\overline S$ 
are connected and hence equipped with canonical path metrics. We could have
replaced the conditions in Definition \ref{enlhom} by first representing the homology class in question as the image of the fundamental 
class under a map $\phi : M \to X$ from a closed oriented $n$-dimensional manifold $M$ to $X$ and 
requiring that $M$ or appropriate covers thereof share the corresponding largeness property.  This indeed works if 
$\phi$ induces a surjection in $\pi_1$. We preferred the above definition because it 
applies as well to homology classes which are not representable by closed manifolds (which  
is relevant in low dimensions), works for fundamental groups which are not 
necessarily finitely presented and last but not least shows that metric largeness properties are 
not linked to differential-topological properties of manifolds, but only to metric properties of simplicial complexes. Of course this will not prevent 
us from applying the results of the present section mainly to the classical case of closed manifolds later on.

We remark that a map defined on a connected simplicial complex $X$ (with the 
path metric) to a metric space is $\varepsilon$-contracting if and only if this 
holds for the restriction of the map to each simplex in $X$. 

We need to show that Definition \ref{enlhom} does not depend on the choice of $S$. This is in fact 
the main technical argument in this section and we will explain it first in detail 
for enlargeable classes. We start with the following basic extension lemma. 

\begin{lem} \label{extend} Let $k, n \geq 1$ be natural numbers and equip the disk $D^k$ and its boundary $\partial D^k$ with 
fixed, but arbitrary piecewise smooth Riemannian metrics (which need not be related).   
Then there are positive constants $\delta$ and $C$ which depend only on the chosen metrics 
on $D^k$ and $\partial D^k$ an on $n$, so that the following holds: Let $0 < \epsilon < \delta$ and let $f : \partial D^k \to S^n$ be 
a piecewise smooth $\varepsilon$-contracting map. Then $f$ can be extended to a piecewise smooth 
$(C \cdot \varepsilon)$-contracting map $D^k \to S^n$.
\end{lem} 

\begin{proof} Because any two piecewise smooth metrics on $D^k$ are in bi-Lipschitz correspondence, it is 
enough to treat the case when the metric on $D^k$ is given in polar coordinates by $dr^2 + r^{2(k-1)} \cdot g$ 
where $g$ is the given piecewise smooth Riemannian metric on $\partial D^k$. 

We can choose a $\delta > 0$ so that for any $0< \varepsilon< \delta$ and any $\varepsilon$-contracting 
piecewise smooth map 
$f : \partial D^k \to S^n$,
 the image of $f$ is contained in a closed hemisphere $D \subset S^n$. 
Using polar coordinates, we identify the hemisphere with the unit ball $D^n \subset \R^n$ 
by a diffeomorphism $\omega: D^n \to D$. This diffeomorphism is bi-Lipschitz with Lipschitz constants 
independent of $f$. In particular there is a constant $C$, independent 
of $\epsilon$ an $f$, so that $\omega^{-1} \circ f : \partial D^k \to D^n$ is $(C \cdot \epsilon)$-contracting. This implies that 
the diameter of the image of $\omega^{-1} \circ f$ is bounded above by $\diam(\partial D^k) \cdot C \cdot \varepsilon$. 

Now map the midpoint $P$ of $D^k$ to some point contained in the image of $ \omega^{-1} \circ f$ 
and extend $\omega^{-1} \circ f$ to $D^k$ linearly along the radial lines joining $P$. This extended map 
is piecewise smooth and using polar coordinates on 
$D^k$ for computing the lengths of piecewise smooth curves in $D^k$  (and the fact that the metric 
on $D^k$ has the special form described above), 
this extended map is $C \cdot  \varepsilon$-contracting 
with a constant $C$ which is independent of $\varepsilon$ and $f$.

Upon composing this map with the Lipschitz map $\omega$, the claim of the lemma follows.
\end{proof}

Returning to Definition \ref{enlhom} let   
$S' \subset S$ be a smaller finite connected $\pi_1$-surjective subcomplex which carries $c$.  
If one of the properties described in Definition \ref{enlhom} holds 
for $S$, then it holds as well for $S'$ by the naturality of $p^!$ and 
because the lifted inclusion $\overline{S'} \hookrightarrow \overline{S}$ 
is $1$-Lipschitz for any connected cover $\overline{X} \to X$.  
Now let $S, S' \subset X$ be two  finite connected $\pi_1$-surjective subcomplexes carrying $c$. Then 
there is a third finite connected $\pi_1$-surjective subcomplex $T \subset X$ carrying $c$ 
and containing $S$ and $S'$. Hence it remains to show that in Definition \ref{enlhom} we may 
pass from $S$ to a  larger finite connected $\pi_1$-surjective subcomplex $T$ of $X$.

Let $\varepsilon > 0$. By assumption there is a connected cover $p : \overline X \to X$ and an almost proper 
$\varepsilon$-contracting map $f_{\varepsilon} : \overline S \to S^n$ 
satisfying $(f_{\varepsilon})_{*} (p^!(c)) \neq 0$. We will show that if $\varepsilon$ is small 
enough, then $f_{\varepsilon}$ can be extended to a $(C\cdot \varepsilon)$-contracting 
almost proper map $\overline{T} \to S^n$ where $C > 0$ is a constant which depends 
only on $S$ and $T$, but not on $\varepsilon$. 

The proof is by induction on the $k$-skeleta $T^{(k)}$ of $T$, where $0 \leq k \leq \dim T$. However, the
start of the induction is a bit involved, because we need to treat the cases $k = 0, 1$ 
simultanously.  

Let us first assume that $T \setminus S$ contains just one vertex $v$. 
Let $\overline V \subset \overline T$ 
be the set of lifts of $v$. For each $\overline v \in \overline V$, let $F(\overline v) \subset 
\overline S$ be the set of all vertices having a common edge with $\overline v$. Note that because 
$\overline T$ is connected and locally finite, the set $F(\overline v)$ is nonempty and finite. 
Furthermore, $\diam F(\overline v)$  (measured with respect to the path metric on $\overline S$) 
is  independent of $\overline v \in \overline V$. Let $F(\widetilde v) \subset \widetilde S$ be 
the subset defined in an analogous fashion as $F(\overline v)$ but with $\overline S$ replaced
by the universal cover $\widetilde S \to S$ (and $\overline v$ by a point $\widetilde v \in \widetilde S$ over $v$) and 
set
\[
   d := \diam F(\widetilde v)
\]
measured with respect to the path metric on $\widetilde S$. Then $d$ is independent of the choice of 
$\widetilde v$ and $\varepsilon$ and furthermore
\[
    \diam F(\overline v) \leq d \, .
\]
Let $e \subset T$ be a fixed 
edge connecting $v$ with a vertex in $S$. For $\overline v \in \overline V$, we  
set $f_{\varepsilon}(\overline v) := f_{\varepsilon}(v_1(\overline e (\overline v)))$ where 
$\overline e (\overline v)$ is the uniqe lift of $e$ containing $\overline v$ and $v_1(\overline e(\overline v))$ 
is the vertex of this lift different from $\overline v$. The extension 
$f_{\varepsilon} : \overline{S \cup \{v\}} \to S^n$ defined in this 
way satisfies $d(f_{\varepsilon} (v_0) , f_{\varepsilon}(v_1)) \leq \max\{d,1\} \cdot \varepsilon$, 
if $v_0$ and $v_1$ are the vertices in some $1$-simplex of $\overline{S \cup T^{(1)}} = p^{-1}(S \cup T^{(1)})$. 
Next, assuming  
that $\max\{d,1\} \cdot \varepsilon < \delta_1$, we extend $f_{\varepsilon}$ to a 
$(\max\{d,1\} \cdot C_1 \cdot \varepsilon)$-contracting map 
$\overline{S \cup T^{(1)}} \to S^n$ using Lemma \ref{extend}.  Here, $\delta_1$ and $C_1$ are given by 
Lemma \ref{extend}  and depend only on $S$ and $T$.

If $T \setminus S$ contains more than one vertex, we apply this process inductively, where 
in each induction step, we pick a vertex in $T$ which has a common edge with some vertex 
in the subcomplex where $f_{\varepsilon}$ has already been defined (note that in each step, this subcomplex 
is connected). In this way, we get (for small enough $\varepsilon$) a $(C_1' \cdot \varepsilon)$-contracting, 
almost proper extension $\overline{S \cup T^1} \to S^n$ of $f_{\varepsilon}$ where $C_1' > 0$ is 
an appropriate constant which just depends on $S$ and $T$. 

Now, if for $k \geq 3$, we have  a $(C_{1}' \cdot C_2 \cdot \ldots \cdot C_{k-1} \cdot 
\varepsilon)$-contracting (respectively, if $k = 2$, a $(C_1' \cdot \varepsilon)$-contracting) 
almost proper extension $\overline{S \cup T^{(k-1)}} \to S^n$ of $f_{\varepsilon}$, 
we extend this map to a $(C_1' \cdot \ldots \cdot C_k \cdot \varepsilon)$-contracting 
almost proper map $\overline{S \cup T^{(k)}} \to S^n$. This is possible by 
Lemma \ref{extend}, if $\varepsilon$ is small enough (where the smallness just depends on $S$, $T$ and $k$). 

A similar argument works for hypereuclidean and hyperspherical classes. 

If we are dealing with coarse largeness properties, the preceding argument can 
be replaced by the simple observation that the inclusion $\widetilde S \to \widetilde T$ 
is a coarse equivalence, where $\widetilde S = p^{-1}(S)$ is the preimage 
of $S$ under the universal covering map $p : \widetilde X \to X$ and similarly for $T$.

As a further consequence of our argument we notice the following slightly different 
characterisation of large classes which will be important for 
one of our  constructions in the next section.

\begin{prop} \label{raffiniert} Let $X$ be a connected simplicial complex with finitely generated fundamental group and 
let $c \in H_n(X;\Q)$ be a homology class. Then $c$ is enlargeable, if and only if the following 
holds: Choose a finite 
connected $\pi_1$-surjective subcomplex 
$S \subset X$ carrying $c$. Let $C \subset S$ be a finite subcomplex 
carrying $c$ which is not necessarily connected or $\pi_1$-surjective.  Then, for 
any $\varepsilon > 0$, there is a connected cover $p : \overline X \to X$ and an $\varepsilon$-contracting almost 
proper map 
\[
    \overline C \to S^n 
\]
mapping $p^!(c)$ to a nonzero class. Here, $\overline C := p^{-1}(C) \subset \overline S= p^{-1}(S)$ is equipped with 
the restriction of the canonical path metric on $\overline S$. 
\end{prop} 

\begin{proof} It is easy to see that the given property is necessary for the enlargeability of $c$. 
Conversely, if this condition is satisfied by $C$ we apply the preceding argument in order to show that it 
is also satisfied by the larger subcomplex  $S$. 
The only difference is that in the beginning of the induction (i.e. for $k = 0,1$)
we work with the restrictions of the path metrics from  $\overline S$ to $\overline C$ and from 
$\widetilde S$ to $\widetilde C$ (for the defintition of $d$). Here $\widetilde C := 
\psi^{-1}(C)$, where $\psi : \widetilde X \to X$ is the universal cover.  
\end{proof} 

We remark that analogues of  Proposition \ref{raffiniert} hold for (coarsely) hypereuclidean, (coarsely) 
hyperspherical or macroscopically large classes. 

Next we study functorial properties of large homology classes. 

\begin{prop} \label{funktoreins} Let $X$ and $Y$ be connected simplicial complexes 
with finitely generated fundamental groups and let $\phi  : X \to Y$ be a continuous 
(not necessarily simplicial) map. Then the following implications hold:
\begin{itemize}
   \item If $\phi$ induces a surjection of fundamental groups and $\phi_*(c)$ is 
         enlargeable, then  $c$ is enlargeable.
   \item If $\phi$ induces an isomorphism of fundamental groups and $c$ is enlargeable, 
         then also $\phi_*(c)$ is enlargeable. 
\end{itemize}
If we are dealing with (coarsely) hyperspherical, (coarsely) hypereuclidean or macroscopically
large classes and if we assume  that $\phi$ induces an isomorphism of fundamental groups, 
then $c$ is large in the respective sense if and only if $\phi_*(c)$ is.
\end{prop}

\begin{proof} First, assume that $\phi_*(c)$ is enlargeable and $\phi$ is surjective on $\pi_1$. 
Let $S \subset X$ be a finite connected $\pi_1$-surjective subcomplex carrying $c$. 
Then $\phi(S)$ is contained in a finite connected $\pi_1$-surjective subcomplex $T \subset Y$ which 
carries $\phi_*(c)$. Because $S$ and $T$ are compact, the map $\phi : S \to T$ is Lipschitz   
with Lipschitz constant $\lambda$, say.  

Let  $\varepsilon > 0$ and choose a connected 
cover $p_Y : \overline Y \to Y$ together with an almost proper  $\varepsilon$-contracting map 
$\overline T \to S^n$ mapping $(p_Y)^!(c)$ to a nonzero class. Let $p_X : \overline X \to X$ be the pull back
of this cover under $\phi$. Because $\phi$ is surjective on $\pi_1$, $\overline X$ is connected and 
we get a map of covering spaces
\[
\xymatrix{ 
      \overline S  \ar[r]^{\overline \phi} \ar[d]_{p_X}    &  \overline T  \ar[d]^{p_Y}\\
        S   \ar[r]^{\phi}                   &     T         }                
\]
which restricts to a bijection on each fibre. In particular, 
the map $\overline \phi$ is proper 
and $\lambda$-Lipschitz (as usual, $\overline S$ is equipped with 
the canonical path metric). Hence, 
if $f_{\varepsilon} : \overline T \to S^n$ is an almost 
proper $\varepsilon$-contracting map, then $f_{\varepsilon} \circ \overline \phi$ is almost proper, 
$\lambda \cdot \varepsilon$ contracting and maps $(p_X)^!(c)$ to a nonzero class. 

Now assume that $c$ is large and $\phi$ induces an isomorphism of fundamental groups. 
By the first part of this proof, we can replace $Y$ by a homotopy equivalent complex 
and hence we may assume that $\phi$ is an inclusion. Let $S \subset X$ 
be a finite connected subcomplex carrying $c$. Then $S$ is also a subcomplex of $Y$ 
and it carries $\phi_*(c)$. Because $\phi$ induces an isomorphism on $\pi_1$, 
each connected cover of $X$ can be written as the restriction of a connected cover of $Y$. This 
shows that $\phi_*(c)$ is also enlargeable. 

Again, the other largeness properties are treated in a similar manner. 
\end{proof}

Proposition \ref{funktoreins} implies the important fact that the large classes 
form  a well-defined subset of $H_*(\Gamma ;\Q)=H_*(B \Gamma ; \Q)$ for each finitely generated group $\Gamma$ in the following sense: 
For each simplicial model $X$ of $B \Gamma$, the large classes form a subset of $H_*(X;\Q)$ and if $X$ and 
$Y$ are two models of $B \Gamma$ and $X \to Y$ is the (up to homotopy unique) homotopy equivalence inducing the 
identity on $\pi_1 = \Gamma$, 
then the induced map in homology identifies the large classes in $H_*(X;\Q)$ and $H_*(Y;\Q)$. 

The next corollary states homological invariance of classical largeness properties. 

\begin{cor}\label{manfd} 
Let $M$ be a closed oriented manifold of dimension $n$. Then $M$ is 
large if and only if $\phi_*[M] \in H_n ( B \pi_1(M) ; \Q)$ is large. 
\end{cor}

\begin{proof} 
This is implied by Proposition \ref{funktoreins}, 
because $M$ is large, if and only if $[M] \in H_n(M;\Q)$ is large.  
\end{proof} 

The next theorem indicates that the subset of non-large classes should 
actually be our main concern.

\begin{thm}\label{subsp_thm} Let $X$ be a connected simplicial complex with finitely 
generated fundamental group. Then the non-large homology classes 
in $H_n(X;\mathbb{Q})$ form a rational vector subspace.
\end{thm}

\begin{proof} The class $0 \in H_n(X;\Q)$ is not large: Take any  connected finite $\pi_1$-surjective 
subcomplex $S \subset X$. Clearly, $S$ carries $0$. By a direct application of Definition \ref{enlhom}
it follows that $0$ is not large. 

It is obvious that if $c \in H_n(X;\Q)$ is not large, then no rational multiple of $c$ is large. 

In order to show that the subset of non-large classes is closed under addition, we need to 
show the following: Let $c , d \in H_n(X;\Q)$ and assume that $c + d$ is large. Then one of $c$ and 
$d$ must also be large. For a proof, let $S \subset X$ be a connected finite $\pi_1$-surjective 
subcomplex carrying $c$ and $d$. Then $S$ also carries $c + d$. Assume that $c+d$ is enlargeable (the 
other largeness properties are easier and left to the reader). 
Let $\varepsilon := \frac{1}{k}$ for 
a natural number $k \geq 1$. Because $c + d$ is enlargeable, there is a connected cover $p : \overline X \to X$ 
and an almost proper $\varepsilon$-contracting map $f_{\varepsilon} : \overline S \to S^n$ mapping 
$p^!(c+d)$ to a nonzero class in $H_n(S^n;\Q)$. This can hold only if either $p^!(c)$ or 
$p^!(d)$ is mapped to a nonzero class. Hence, for infinitely many $k$, either $p^!(c)$ or 
$p^!(d)$ is mapped to a nonzero class (for appropriate covers $p : \overline X \to X$) and consequently 
either $c$ or $d$ is enlargeable. 
\end{proof} 

\begin{defn} If one largeness property $P$ is chosen and $X$ is a connected simplicial complex with finitely 
generated fundamental group, we denote the rational vector subspace of $H_*(X;\Q)$ consisting 
of classes which are not large with respect to $P$ by $H^{\rm sm(P)}_*(X;\Q)$. This is the {\em small homology of $X$ with respect to $P$} and depends a priori on 
the given largeness property $P$. 
\end{defn} 

Theorem \ref{subsp_thm} implies that $0 \in H_n^{\rm sm(P)}(B \Gamma;\Q)$ for each finitely generated 
group and each largeness property $P$. Together with Corollary \ref{manfd} this shows again that large manifolds are 
essential. 

The results of this section leave it  as a central problem to determine $H^{\rm sm(P)}(B \Gamma;\Q)$
 for different largeness properties and finitely generated groups $\Gamma$. This will be the topic of the next section.

\section{Examples and applications} \label{bspiele}

The following theorem illustrates the usefulness of our systemtic approach to large homology 
in Section \ref{essential_and_homol_inv}.

\begin{thm} \label{calculation} 
The small homology of $B \Z^k = T^k$, $k \geq 1$,  is calculated as follows. 
\begin{itemize}
   \item If $P$ denotes enlargeability we have $H_n^{\rm sm(P)}(T^k;\Q) = 0$ for all $n \geq 1$. 
   \item If $P$ denotes (coarse) hypereuclideaness, (coarse) hypersphericity and macroscopic largeness  we have 
   \[
   H_n^{\rm sm(P)}(T^k;\Q) = 
   \begin{cases}
   0 & \textrm{for~} n = k  \\
   H_n(T^k;\Q) & \textrm{for~} 1 \leq n < k  
   \end{cases}. 
   \] 
\end{itemize}
\end{thm}

\begin{proof} We equip $T^k = \R^k / \Z^k$ with  the metric induced from $\R^k$.  
Let $0 \neq c \in H_n(T^k;\Q)$. 
We write
\[
    c = \sum_{1 \leq i_1 <  i_2 <  \ldots < i_n \leq k} \alpha_{i_1, \ldots, i_n} t_{i_1, \ldots, i_n} 
\]
where $t_{i_1, \ldots ,i_n } \in H_n(T^k;\Z)$ is represented by the embedding 
\begin{align*}
    T^n & \to T^k, \\
   (\xi_1, \ldots, \xi_n) & \mapsto (1, \ldots, 1, \xi_{1}, 1, \ldots, 1, \xi_n, 1, \ldots ,1) 
\end{align*} 
of the $n$-torus into $T^k= (S^1)^k$, where $\xi_{\nu}$ is put at the $i_{\nu}$th entry, and each $\alpha_{i_1, \ldots, i_n} \in \Q$. 
Without loss of generality  we 
may assume that $\alpha_{1,\ldots, n} \neq 0$. 

Let $p : \R^n \times T^{k-n} \to T^k$ be the cover associated to the subgroup $\Z^{k-n} = 0 \times \Z^{k-n} \subset \Z^k$ 
given by the last $k-n$ coordinates. For $\varepsilon > 0$, let $f : \R^n \to S^n$ be an $\varepsilon$-contracting 
almost proper map of nonzero degree. Then the composition 
\[
    f_{\varepsilon} : \R^n \times T^{n-k} \stackrel{pr_{\R^n}}{\longrightarrow} \R^n \stackrel{f}{\longrightarrow} S^n
\]
is $\varepsilon$-contracting, almost proper and   
\[
    (f_{\varepsilon})_*( p^!(t_{i_1, \ldots, i_n})) \neq 0 \in \widetilde H_n(S^n;\Q) 
\]
if $(i_1 ,\ldots, i_n) = (1,2, \ldots, n)$ and is zero otherwise. Hence, 
$(f_{\varepsilon})_*(p^!(c)) \neq 0$. This shows that $c$ is enlargeable.

Concerning the other largeness properties it is clear that $H_k^{\rm sm(P)}(T^k;\Q) = 0$, because the  
universal cover $\R^k$ of  $T^k$ shares all of the mentioned largeness properties. Futhermore, notice 
that the transfer maps each non-zero class in $H_k(T^k;\Q)$ to a rational multiple of the locally 
finite fundamental class of $\R^k$.

Now let $1 \leq n < k$ and let $c \in H_n(T^k;\Q)$. The full space $T^k$ carries $c$ and is
$\pi_1$-surjective. However, if $n < k$, then Poincar\'e duality implies 
\[
      H^{\rm lf}_{n}(\R^k;\Q) \cong H^{k-n}(\R^k;\Q) = 0
\]
so that under the universal cover $p : \R^k \to T^k$, the transfer $p^!(c) \in H^{\rm lf}_n(\R^k;\Q)$ is equal 
to zero. This shows that $c$ cannot be large. 
\end{proof} 

Combined with Corollary \ref{manfd} this calculation has the following interesting consequence. 

\begin{thm} \label{interesting} Let $M$ be a closed orientable manifold of dimension $n$ and with fundamental group $\Z^k$, 
where $1 \leq n < k$. Then $M$ is enlargeable if and only if it is essential. However, the 
universal cover of  $M$ is never (coarsely) hypereuclidean, (coarsely) hyperspherical or macroscopically large. 
\end{thm} 

For $4 \leq n < k$ we can construct essential $n$-dimensional manifolds $M$ with 
fundamental group $\Z^k$ as follows. Start with the oriented connected sum
\[
    C := T^n \sharp \big( \sharp^{k-n}  (S^1 \times S^{n-1}) \big) 
\]
of an $n$-torus with $k-n$ copies of $S^1 \times S^{n-1}$.  This manifold has fundamental group $\pi_1(C) = 
\Z^n *  (*^{k-n} \Z)$, the free product of $\Z^n$ with $k-n$ copies of $\Z$. 
Let $C \to B\pi_1(C)$ be the classifying map of the universal cover and consider
the composition $\phi : C \to B\pi_1(C) \to B\Z^k$ induced by the abelianization $\Z^n * (*^{k-n} \Z) \to \Z^k$. 
Then $\phi$ sends the fundamental class of $C$ 
to a non-zero class in $H_n(T^k ; \Z)$. Moreover it  induces a surjective 
map of fundamental groups $\pi_1(C) \to \pi_1(B \Z^k) = \Z^k$ whose kernel 
can be killed by oriented surgeries in $C$ (here we use the assumption $n \geq 4$). 
The resulting manifold $M$ has the stated properties.

Together with Theorem \ref{interesting} this completes the proof of Theorem \ref{example1}.

We will now describe a refined  construction to obtain essential manifolds which are not enlargeable. As
we have to deal with arbitrary covers of the manifolds in question, we need to recall some facts 
from covering space theory. 

Let $X$ be a path connected space and let $S \subset X$ be a path connected subspace. 
We choose a base point $\widetilde x$ in the universal 
cover $\widetilde X$ of $X$ which lies over $S$ and 
compute all fundamental groups with 
respect to $\widetilde x$ and its images in the different covers of $X$. Let $G = \pi_1(X)$, let $H \subset G$ 
be a subgroup and let $\lambda : \pi_1(S) \to \pi_1(X)$ be the map induced by the inclusion $S \subset X$. We denote 
the image of this map by $\Sigma \subset \pi_1(X)$. 
Let $p : \overline X \to X$ be the cover 
associated to $H$. 
Recall that $G$ acts on $\widetilde X$ from the left in a canonical way. 

In view of Proposition \ref{raffiniert} we collect some information on the components 
of $\overline S := p^{-1}(S)$. The projection $\widetilde X \to \overline X$ is denoted by $\psi$. 
If $g_1, g_2 \in G$, then $g_1 \widetilde x$ and $g_2 \widetilde x$ are mapped to the same component of $\overline S$ 
if and only if there are elements $\sigma \in \Sigma$ and $h \in H$ with $g_1 \sigma = h g_2$. Hence 
the components of $\overline S$ are in one to one correspondence with double cosets $H \backslash G / \Sigma$. 
If $P \subset \overline S$ is one component, then there exists a $g \in G$ so that $\psi(g \widetilde x) \in P$. 
With respect to this base point, the fundamental group of $P$ is canonically isomorphic 
to $\lambda^{-1} (g^{-1} H g) \subset \pi_1(S)$. 
Replacing $g \widetilde x $ by $h g \sigma \widetilde x$ gives another point 
whose image lies in $P$, and with respect to this base point the fundamental group of $P$ is equal 
to $\lambda^{-1}(\sigma^{-1} g^{-1} H g \sigma)$. 

If $\lambda$ is injective, we identify $\pi_1(S)$ and $\Sigma$ and 
then the fundamental group of the component $P \subset \overline S$ with respect to the 
base point $\psi(g \widetilde x)$ is equal to $\Sigma \cap (g^{-1} H g)$. If moreover $\Sigma$ 
is normal in $\pi_1(X)$, this group is equal to $g^{-1}(\Sigma \cap H) g$ and so the 
fundamental groups of the components $P \subset \overline S$ are mutually isomorphic. 

Let 
\[
   \Hig := \langle a,b,c,d | a^{-1} b a = b^2, b^{-1} c b = c^2, c^{-1} d c = d^2, d^{-1}a d = a^2 \rangle
\]
be the Higman $4$-group \cite{Hig}. This is a finitely presented group with no proper subgroups of finite index. 
By \cite{BDH} ${\rm Hig}$ is integrally acyclic, i.\,e.\ $\widetilde H_*(\Hig ; \Z) = 0$. Pick a  
$z \in \Hig$ of infinite order. We define 
\[
     K:=  \Hig *_{\langle z \rangle} \rm \Hig
\]
as the amalgamated free product of $\Hig$ with itself along $z$. We claim that the group $K$ still does not possess any proper 
subgroups of finite index. Assume the contrary and let $H < K$ be a proper subgroup of finite index. Then the 
homomorphism from $K$ to the permutation group of $K/H$ induced by the left translation action of $K$ 
is nontrivial and has finite image. But then the push out property of the amalgamated free product implies that also $\Hig$ has a nontrivial 
homomorphism to a finite group, contradicting the fact that $\Hig$ has no proper subgroups of finite index. 

A Mayer-Vietoris argument shows that $\widetilde H_*(K;\Z) = 
\widetilde H_*(S^2; \Z)$. Let $c \in H^2(K;\Z)$ 
be a generator and define 
\begin{itemize}
    \item the group $L$ by the central extension $1 \to \Z \to L \to K \to 1$ classified by $c$  and
    \item  the group $N$ by the central extension $1 \to \Z \to N \to \Z^2 \to 1$ 
              classified by a nontrivial generator of $H^2(\Z^2; \Z) \cong \Z$. 
\end{itemize}
Notice that $N$ is the fundamental group of the closed oriented $3$-manifold arising as the 
total space of the $S^1$-bundle over $T^2$ with Euler number $1$. This can be regarded as
the quotiend $\Nil/N$ of the corresponding $3$-dimensional nilpotent Lie group by the cocompact lattice
$N$. 

We consider $\Z$ as a central subgroup of $N \times L$ via the diagonal embedding and finally set 
\[ 
    G := (N \times L) / \Z \, . 
\]

In the following, we will regard 
\begin{itemize}
   \item $N$ as a normal subgroup of $G$ via the inclusion $N = N \times  1  \subset N \times L \to G$ and 
   \item $\Z$ as a central subgroup of $G$ via the inclusion $\Z \subset N \subset G$. 
\end{itemize}

The following formulates a key property of $G$.

\begin{thm} \label{besonders} Let $H \subset G$ be a subgroup which maps surjectively onto $K$ under the 
canonical map 
\[
   p : G \to N/\Z \times L/\Z \to L / \Z = K \, . 
\]
Then the generator $e \in \Z\subset G$ belongs to $H$. 
\end{thm} 

\begin{proof} Restricting $p$ to the subgroup $H$ leads to an extension 
\[
    1 \to F \to H \to K \to 1 
\]
where $F = \ker p|_H \subset N$.  Let
\[
   \phi : G \to N/ \Z \times L / \Z \to N /  \Z = \Z^2
\]
be the canonical map. 

First, assume $\phi(F)  = 0$. This implies that the canonical map 
\[
    G \to N/ \Z \times L/ \Z = \Z^2 \times K 
\]
sends $H$ to a subgroup $\overline{H}  \subset \Z^2 \times K$ which maps isomorphically to $K$ 
under the projection onto the second factor. Because any homomorphism of $K$ to a finitely 
generated abelian group is trivial (using the fact that $K$ has no proper subgroups of finite index), 
we conclude that the projection of $\overline{H}$ onto $\Z^2$ is trivial. In summary this means that 
$1 \to F \to H \to K \to 1$
is a subextension of $1 \to \Z \to L \to K \to 1$. 
Because $c \in H^2(K;\Z)$ is indivisible, this implies that $F = \Z$ and hence
$e \in H$ as desired. 

Now we consider the case $\rk \phi(F) > 0$. Let  
\[
   \psi :  G \stackrel{\phi}{\to}  \Z^2 \to \Z 
\]
be the composition of $\phi$ with one of the projections $\Z^2 \to \Z$ so that $\psi(F) \neq 0$. 
Then $\psi(F)$ is a finite index subgroup of $\psi(H)$. 
Hence $H' := \psi^{-1}(\psi(F)) \subset H$ still maps surjectively to $K$, as $K$ has no proper subgroups of finite 
index. Because $\psi(H') = \psi(F)$ and $F \subset \ker p$,  this implies that $\ker \psi \subset H$ 
also maps surjectively to $K$. Repeting this argument if necessary this shows that we can 
assume $\phi(F) = 0$ and we are in the case treated before.  
\end{proof} 

\begin{lem} \label{comp} 
We have $H_3(N; \Z) \cong \Z$. Moreover,
the inclusion $N \to G$ induces an isomorphism $ H_3(N;\Z) \cong H_3(G;\Z)$. 
\end{lem} 

\begin{proof} The three dimensional closed oriented manifold $\Nil/N$ is a 
model for $BN$. This implies the first assertion. 

The homological spectral sequence for the central extension 
\[
   1 \to \Z \to L \to K \to 1 
\]
shows that $ H_*(L;\Z)\cong \Z$ in degree $0$ and $3$ and $H_*(L;\Z) =0$ otherwise. 
With this information we conclude that the spectral sequence for the normal extension 
\[  
   1 \to L \to G \to N/\Z \to 1 
\]
collapses at the $E^2$-level (recall that $N/\Z = \Z^2$). 
Because the induced action of $N/\Z$ on $L$ is trivial, this 
implies that $H_*(G; \Z)$ is free of rank $1,2,1,1,2,1,0, \ldots$ in degrees $0,1,2,3,4,5,6, \ldots$. 
With this information we go into the spectral sequence for the normal extension 
\[
   1 \to N \to G \to L/\Z \to 1 
\]
and conclude (using that the induced action of $L/\Z$ on $N$ is trivial) that on the 
$E^2$-level the differential 
\[   
   \partial^2 : E^2_{2,1} \to E^2_{0,2} 
\]
is an isomorphism (of free abelian groups of rank $2$) and all other differentials (on $E^2$ or on higher levels) 
are zero. This implies that $E^2_{0,3}= H_0(L/\Z; H_3(N;\Z)) \cong \Z$ cannot be hit by a differential and 
hence the assertion of Lemma \ref{comp}. 
\end{proof} 

\begin{thm} \label{nonenl} 
The homology group $H_3(G;\Q)$ (which is different from zero by Lemma \ref{comp}) consists only of non-enlargeable classes. 
\end{thm} 

Before we go into the proof, we deal with the following lemma concerning the manifold $C := \Nil / N$ considered before. 
We fix a basepoint $c \in C$. 

\begin{lem} \label{referee} Let $C$ be equipped with some Riemannian metric. 
Then there is an $\varepsilon > 0$ so that the following holds: 
Let $P \to C$ is a connected cover which is equipped with the lifted metric, and $p \in P$ is a point over $c$, 
then an $\varepsilon$-contracting almost
proper map $P \to S^3$ of nonzero degree can only exist, if the image of $\pi_1(P,p)$ in $\pi_1(C,c)$ 
does not contain $e \in \Z \subset N = \pi_1(C,c)$.
\end{lem} 

\begin{proof} Let $P \to C$ be a connected cover so that $\pi_1(P,p)$, considered as a subgroup of $\pi_1(C,c)$, 
contains the generator $e \in N$. Then $P$ can be written as the total space of the pull back of the bundle 
$S^1 \to C \to T^2$ along a covering map  $\phi: V \to T^2$. 

Let $D^2 \to D(C) \to T^2$ be the disk bundle of $S^1 \to C \to T^2$ and extend the given Riemannian metric on $C$ to $D(C)$. 
Let $D^2 \to D(P) \to V$ be the disk bundle of $S^1 \to P \to V$ equipped with the pull back metric from $D(C)$. 

There is a relative CW-structure on the pair $(D(C), C)$ with finitely many cells attached along 
piecewise smooth maps. Hence the Riemannian manifold $D(P)$ can be obtained from $P$ by 
attaching lifts of the same cells equipped with Riemannian metrics induced from $D(C)$. 

Using Lemma \ref{extend} inductively over the cells in this relative CW-decomposition of $(D(P), P)$ 
we find an $\epsilon > 0$, which is independent of
$\phi$ (i.e. of the specific cover  $P$),  so 
that any almost proper $\varepsilon$-contracting map $P \to S^3$ can be extended to an almost  proper map $D(P) \to S^3$.   This implies
that any $\varepsilon$-contracting almost proper map $P \to S^3$ must be of zero degree. 
\end{proof} 

\begin{proof}[Proof of Theorem \ref{nonenl}.] 

We write $X := BG$. Let $\Nil/N = BN \to X$ be the map induced by the inclusion $N \hookrightarrow G$. 
We assume that $BN$ is a finite simplicial complex and that this map is an inclusion of simplicial complexes. 
By Lemma \ref{comp} it is enough 
to show that the image $c\in H_3(X;\Q)$ of the fundamental class $[BN]$ is not enlargeable. Our proof 
is based on Proposition \ref{raffiniert}. Choose a finite 
connected $\pi_1$-surjective subcomplex $S \subset X$ carrying $c$.  We can assume that  $C:=BN$ is contained in $S$. 
Note that $C$ is not $\pi_1$-surjective. Because $G$ is finitely presented, we can and will assume that the inclusion
 $S \subset X$ induces an isomorphism on $\pi_1$.

Using Lemma \ref{referee} there is a constant $\varepsilon > 0$ 
with the following property: If $P  \to C$ is a connected cover and $P$ is equipped with a metric 
which is dominated by the simplicial path metric, then an $\varepsilon$-contracting almost proper map $P \to S^3$ 
of nonzero degree can only exist, if the fundamental group of $P$ does not contain $e \in \Z  \subset N = \pi_1(C)$. 
We fix an $\varepsilon$ with this property and 
- in view of Proposition \ref{raffiniert} - assume that we are given a connected 
cover $p : \overline X \to X$ and an almost proper $\varepsilon$-contracting map
\[
    f_{\varepsilon} : \overline C \to S^3
\]
so that $(f_{\varepsilon})_*((p|_{\overline C})^!(c)) \neq 0$, where $\overline C := p^{-1}(C)$ is equipped with the metric 
induced from $\overline{S} := p^{-1}(S)$. We can choose a 
component $P \subset \overline C$ so that $(f_{\varepsilon})_*((p|_P)^!(c)) \neq 0$. 

Our aim is to show that $f_{\varepsilon}$ has 
nonzero degree on infinitely many components of $\overline C$ and therefore cannot be almost proper, which contradicts our 
assumptions.  
Roughly speaking this is implied by the fact that $\overline C$ contains a string of 
infinitely many components ``parallel'' to $P$. 

Let $\widetilde x \in \widetilde X$ 
be a base point lying over  $P$ and set $H:= p_*(\pi_1(\overline X)) \subset G$ where we choose images of $\widetilde x$ as the 
base points of $\overline X$ and of $X$. 
If under the projection $\phi : G \to K$
the subgroup $H$ maps surjectively to $K$, then 
it follows from Theorem \ref{besonders} that $e \in N \cap H = \pi_1(P)$. 
By the choice of $\varepsilon$, this 
is impossible and hence $H$ cannot map surjectively onto $K$. Because $K$ has no proper subgroup of finite 
index, we conclude $[K:\phi(H)] = \infty$. The set of double cosets $H \backslash G/ N$ is mapped by $\phi$ to the 
set of cosets $\phi(H)\backslash K$ and is hence infinite. Choose a finite set $\Sigma \subset L \subset G$ mapping
to a set of generators of $K$. Because $[K : \phi(H)] = \infty$, there is a sequence $(\sigma_n)_{n \in \N}$ of elements in $\Sigma$
so that the set $\{ \sigma_{1} \sigma_{2} \ldots \sigma_{k} \, | \, k \in \N\setminus \{0\} \}$ maps to an infinite 
set of cosets in $\phi(H) \backslash K$. 

Now all elements in $\Sigma$ commute with all elements in  $N$. 
However, the elements of $\Sigma$ do not act on $\overline S$, because $\Sigma$ need not be 
in the normalizer of $H=\pi_1(\overline S)$. 
But certainly each $\sigma \in \Sigma$ acts as a deck transformation on the universal cover $\psi : \widetilde S \to S$. 
We set $\widetilde C := \psi^{-1}(C) \subset {\widetilde S}$ and denote by $\widetilde P \subset \widetilde C$ the component 
containing $\widetilde x$. Note that $\widetilde P \to C$ is a universal covering.  Let  $\tau_k : \widetilde S \to \widetilde S$ 
be the deck transformation associated to $\sigma_1 \sigma_2 \ldots \sigma_k$ and set 
$\tau_0 := \id_{\widetilde S}$. Since $\Sigma$ is finite and $G$ acts isometrically on $\widetilde S$,
there is a positive number $\Delta$ so that 
\[
   d_{\widetilde S}(\tau_{k} (x), \tau_{k+1} (x)) = d_{\widetilde S} ( \tau_k(x) , \tau_k(\sigma_{k+1}(x)) = 
   d_{\widetilde S}(x, \sigma_{k+1}(x)) \leq \Delta 
\]
for all $x \in \widetilde P$  and all $k \in \N$. For the last inequality we use the facts that the actions of $N$ and $L$ commute and 
that $\widetilde P$ is invariant under the restricted action of $N$ with compact quotient $C$ (and hence with compact 
fundamental domain). Note that $\Delta$ is independent of
$H$ and of $\varepsilon$. We will henceforth assume that $\varepsilon$ has been chosen so small 
that $\varepsilon \cdot \Delta < \pi/2$.

Let $q : S' \to S$ be the cover associated to the subgroup $N \cap H \subset H$. Under the covering 
map $r: S' \to \overline S$ (recall that the cover $\overline{S} \to S$ was associated to $H \subset G$), 
each component of $C':=q^{-1}(C)$ 
maps bijectively to a component of $\overline C$ by a map of Lipschitz constant $1$ (here 
we use the metrics from $S'$ and $\overline S$, respectively). Let $P'\subset C'$ 
be the component of the image of $\widetilde x$. Now 
each $\sigma \in \Sigma$ acts on the covering $S' \to S$, because $\Sigma$ is contained in the centralizer of $N \cap H$. 
Furthermore we know that $d_{S'}(\tau_{k}(x) ,\tau_{k+1}(x)) \leq \Delta$ for all $x \in P'$ and all $k \in \N$ (this 
uses the fact that the projection $\widetilde S \to S'$ has Lipschitz constant $1$).

The compositions 
\[
        \tau_k(P')  \stackrel{r}{\to} \overline S \stackrel{f_{\varepsilon}} \to S^3
\]
for increasing $k$ have the same degree, different from zero, because for adjacent $k \in \N$ the respective 
maps are   $\pi/2$-close to each other (after identifying $\tau_k (P') = P' = \tau_{k+1}(P')$) 
and for $k = 0$ we see 
the map $P' \to P \to S^3$ which has nonzero degree by assumption. By the choice of the sequence 
$(\sigma_n)$, this means that $f_{\varepsilon}$ is nonconstant on infinitely many components of $\overline C$ and 
therefore not almost proper.  
\end{proof}

Before we can prove Theorem \ref{example2} we need the following extension of Theorem \ref{besonders}. 

\begin{prop} \label{extension2} Let $G= (N \times L)/\Z$ be the group considered in Theorem \ref{besonders}, let $k \in \N$ and let 
\[
    H \subset   G \times \Z^k 
\]
be a subgroup which maps surjectively onto $K$ under the canonical map 
\[
   p : G \times \Z^k \to ( N / \Z \times L / \Z)  \times \Z^k \to L/\Z = K \, . 
\]    
Then $(e,0, \ldots, 0) \in H$ where $e$ is a generator of $\Z \subset N \subset G$ as before.  
\end{prop}

\begin{proof} We adapt the proof of Theorem \ref{besonders} accordingly. 

We set $F := \ker p|_H$ and distinguish two cases for $\phi(F) \subset (\Z^2) \times \Z^k$ where 
\[
     \phi : G \times \Z^k \to (N / \Z \times L/ \Z) \times \Z^k \to N/\Z \times \Z^k = \Z^2 \times \Z^k 
\]
is the canonical map.

If $\phi(F) = 0$, then we conclude similarly as before that the image of $H$ in $(N/\Z \times L / \Z) \times \Z^k = (\Z^2 \times K) \times \Z^k$ 
projects onto the trivial subgroup in $\Z^2 \times \Z^k$. Hence we can again regard 
$1 \to F \to H \to K \to 1$ as a subextension of $1 \to \Z \to L \to K \to 1$ 
and the indivisibility of the 
class $c \in H^2(K; \Z)$ implies  $F = \Z \times 0 \subset N \times \Z^k$ as desired. 

If $\phi(F) \neq 0$, then we consider one of the compositions 
\[
   \psi : G \times \Z^k \stackrel{\phi}{\to} \Z^2 \times \Z^k \to \Z 
\]
with $\psi(F) \neq 0$. Arguing as before, the subgroup $\ker \psi \subset H$ still maps surjectively onto $K$. Repeting this 
argument we are reduced to the case $\phi(F) = 0$. 
\end{proof}

\begin{thm} For each $k \geq 0$, we have $H_{3+k}(G \times \Z^k;\Z) \cong \Z$ and $H_{3+k}(G \times \Z^k;\Q)$ 
consists only of non-enlargeable classes. 
\end{thm}

\begin{proof}  The proof is analogous to the one of Theorem \ref{nonenl}. The main difference is that $X$ is 
replaced by $B(G \times \Z^k) = BG \times T^k$, the subgroup $N$ by $N \times \Z^k$  and $C$ by 
the subcomplex $\Nil/N \times T^{k} \subset BG \times T^k = X$. In this context, Lemma \ref{referee}, with 
$e \in \Z \subset N$ replaced by $(e, 0 , \ldots, 0) \in \Z \times \Z^k \subset N \times \Z^k$ and 
$S^3$ replaced by $S^{3 +k}$, remains valid so that 
Proposition \ref{extension2} allows us to adopt the proof of Theorem \ref{nonenl}. 
\end{proof}

For the proof of Theorem \ref{example2}, let $n \geq 4$. We can construct an essential $n$-dimensional 
manifold $M$ with fundamental group $G \times \Z^{n-3}$ by first representing the homology class 
$[\Nil/N] \times [T^{n-3}] \in H_{n}(B (G \times \Z^{n-3}) ; \Q)$ 
as $\phi_*[M]$ with a closed oriented manifold $M$ and some map $\phi : M \to B (G \times \Z^{n-3})$ and then improving the map $\phi$ 
by surgery so that it induces an isomorphism on $\pi_1$ (without changing $\phi_*[M]$). This 
is possible by the definition of $G$ and because $L$ is finitely presented. The previous corollary says that 
the class $[\Nil/N] \times [T^{n-3}]$
is not enlargeable. Now the claim follows from Corollary \ref{manfd}.

\section{Higher enlargeability implies essentialness}\label{higher_enl}

This section contains the proof of Theorem \ref{intro_p4_higher_enl_impl_ess}. We argue by induction, 
where the inductive argument relies on a bound, shown by Gromov, on the radius of a submanifold by its volume.

The notion of $k$-dilation extends in an obvious way to piecewise smooth maps of simplicial complexes with 
piecewise smooth Riemannian metrics. For triangulated Riemannian manifolds this new definition 
coincides with Definition \ref{dilation}. Moreover the inequality 
\[ 
  \dil_l(f)^{1/l } \leq \dil_k(f)^{1/k} 
\]
for $l\geq k$ remains valid.
 
\begin{defn} \label{higherenl} Let $X$ be a connected finite simplicial complex equipped with smooth Riemannian metrics on 
each of its simplices. Let $c \in H_n(X;\Q)$ be a homology class. The class $c$ is 
called {\em $k$-enlargeable}, if for every $\varepsilon>0$
there exists a connected cover $p : \overline X_\varepsilon \to X$ and a piecewise smooth almost proper map 
$f_\varepsilon: \overline X_\varepsilon  \to S^n$ with $k$-dilation at most $\varepsilon$ 
and satisfying $(f_{\varepsilon})_*(p^!(c)) \neq 0$.
\end{defn}

Contrary to Definition \ref{enlhom} we allow metrics on the simplices in $X$ 
which are different from the standard metrics. This flexibility is 
convenient for the following argument. 
By the compactness of $X$, Definition \ref{higherenl} is independent of the choice 
of the Riemannian metric on the simplices of $X$.

\begin{prop}\label{extension}
Let $X$ be a connected finite simplicial complex, and let $l \leq n$ be positive integers. Let 
$c \in H_n(X;\Q)$ be $l$-enlargeable. Let $X'$ be obtained from $X$ 
by attaching  
finitely many $(l+1)$-cells (i.e. simplicial $(l+1)$-balls) to the $l$-skeleton $X^{(l)}$ of $X$. 
Then the image of $c$ in $H_n(X';\mathbb{Q})$ is $(l+1)$-enlargeable.
\end{prop}

Note that in contrast to the extension Lemma \ref{extend} (which corresponds to the case $l=1$
in Proposition \ref{extension}) and the subsequent argument, 
we must - at least for $l \geq 2$ - pass from the $l$-dilation to the $(l+1)$-dilation 
upon attaching $(l+1)$-cells.

In the proof we will need the following lemma.

\begin{lem}\label{bound_on_Rad}
There exists a constant $C_n>0$ depending only on $n$ such that for any piecewise smooth map $f:N\to S^n$ from an 
$l$-dimensional manifold $N$ with $l<n$ to the unit $n$-sphere there is a piecewise smooth map $f':N\to S^n$ to an 
$(l-1)$-dimensional 
subcomplex of $S^n$ such that $d(f(x),f'(x))\leq C_n\cdot\Vol_l(f(N))^{1/l}$ for all $x\in N$.
\end{lem}

Gromov \cite[Proposition 3.1.A]{Gromov(1983)} showed a similar statement for the Euclidean space instead of the unit sphere. 
The above lemma may be proved in an analogous manner or can easily be deduced from Gromov's result.

\begin{proof}[Proof of Proposition \ref{extension}]
Let $g$ be a Riemannian metric on $X$. Since the attached cells do not interfere with each other we may assume
that there is only one $(l+1)$-cell to attach. Let $h:S^l \to X$ be the 
(simplicial) attaching map. Choose a lift $\widetilde h:S^l\to \widetilde X$ to the universal covering, 
and denote the $l$-dimensional volume of the image $\widetilde h(S^l)$ by $v$. 

First assume $l<n$. Let $\varepsilon>0$. Choose $\delta>0$ such that $C_n \cdot v^{1/l} \cdot 
\delta^{\frac{l+1}{l}} \leq \varepsilon$ 
and such that $\delta^{\frac{l+1}{l}}\leq\varepsilon$. Let $f_\delta: \overline X_\delta \to S^n$
be almost proper with $(f_{\delta})_*(\overline c) \neq 0$ and $\dil_l(f_\delta)\leq \delta$, where $\overline c$ is
 the transfer of $c$ to $\overline{X}_\delta$. 
Note that 
\[ 
\dil_{l+1}(f_\delta) \leq \dil_l (f_\delta)^{\frac{l+1}{l}} \leq \delta^{\frac{l+1}{l}} 
\leq \varepsilon. 
\]

We extend the given Riemannian metric $g$ on $X$ over the attached $(l+1)$-cell as follows: Think of $X'$ as
\[ 
       X \;\;\;\bigcup_h\;\;\; S^l\times [-1,0] \;\;\;\bigcup\;\;\; S^l\times[0,1] \;\;\;\bigcup\;\;\; S^{l+1}_+ 
\]
and define the metric by taking
\[ 
       g, \;\;\; (-t h^* g + (1+t) g_r) +  dt^2,\;\;\; g_r  +  dt^2,\;\;\; g_r 
\]
on the respective parts. Here, $g_r$ is the round metric of radius $r$ on $S^l$, respectively on the hemisphere 
$S^{l+1}_+$, 
with $r$ chosen so that $h:(S^l,g_r)\to (X,g)$ is $1$-contracting.

Moreover, we extend $f_\delta$ over the attached cells as follows: on $S^l\times [-1,0]$ we use the projection to 
$S^l$ and apply the composition $f_\delta \circ \overline{h}$ where $\overline{h}: S^l  \to \overline{X}_{\delta}$ is an 
appropriate lift of $h$. The $l$-dimensional volume of
 $f_\delta\circ \overline{h}(S^l)$ is at most $\delta \cdot  v$. By Lemma \ref{bound_on_Rad} there is a map $f':S^l\to S^n$ to an
 $(l-1)$-dimensional subcomplex of the $n$-sphere such that $d(f',f_\delta\circ \overline{h})\leq C_n (\delta \cdot v)^{1/l}$. 
The cylinder lines $\{x\}\times [0,1]$ are mapped to minimizing geodesics from $f_\delta\circ \overline{h}(x)$ to $f'(x)$. For 
small enough $\delta$ this map is  well-defined. The remaining cap $S^{l+1}_+$ may be regarded as the 
cone over $S^l\times \{1\}$ and is mapped to some cone over the $(l-1)$-dimensional subcomplex to which 
$S^l\times \{1\}$ is mapped.

By the choice of $\delta$, this new map has $(l+1)$-dilation at most $\varepsilon$: on $S^l\times [-1,0]$ 
and the cap $S^{l+1}_+$, because they are mapped to $l$-dimensional subcomplexes, which 
are zero sets for the $(l+1)$-dimensional
 volume, and on $S^l\times [0,1]$ because the $l$-dimensional volume of the first factor is multiplied by a factor of at most 
$\delta$ and the second factor is $C_n (\delta \cdot v)^{1/l}$-contracted. Hence the 
image of $c$ in $H_n(X';\mathbb{Q})$ is $(l+1)$-enlargeable.

Finally assume $l=n$. For any $\varepsilon>0$ satisfying $\varepsilon v < \Vol_n(S^n)$ the composition 
with $f_\varepsilon$ of any lift $\overline{h}$ of the attaching map is not surjective. Hence, 
$f_\varepsilon\circ \overline{h}$ is nullhomotopic and we may extend $f_\varepsilon$ over $\overline X_\varepsilon$
with the new cells attached.  
Since $S^n$ is $n$-dimensional, every map to it has zero $(n+1)$-dilation.
\end{proof}

Next, we will show Theorem \ref{intro_p4_higher_enl_impl_ess} by an inductive argument. In this proof, 
Proposition \ref{extension} will serve as the induction step.

\begin{proof}[Proof of Theorem \ref{intro_p4_higher_enl_impl_ess}]
Let $M$ be $k$-enlargeable, and let  $\pi_i(M)$ be trivial for $2 \leq i\leq k-1$. Then it is possible to 
construct $B\pi_1(M)$ from $M$ by attaching only cells of dimension at least $k+1$. We may assume 
that the image of the attaching 
map of every $l$-cell lies in the $(l-1)$-skeleton. 

If $M$ is not  essential, then there is a finite subcomplex $X\subset B\pi_1(M)$ 
containing $M$ such that $[M]=0$ 
in $H_n(X;\mathbb{Q})$. We may assume that $X$ is of dimension $n+1$. Then by an induction using Proposition \ref{extension},
 the class $[M] \in H_n(X;\mathbb{Q})$ is $(n+1)$-enlargeable. 
But this contradicts 
the fact that $[M]$ vanishes in $H_n(X;\Q)$. Therefore, $M$ has to be  essential. 
\end{proof}

\end{document}